\documentclass[preprint,12pt]{elsarticle}

\usepackage{graphicx}
\usepackage[cp1251]{inputenc}
\usepackage{amssymb}
\usepackage{amsthm}
\usepackage{cmap}

\usepackage{amsmath}

\biboptions{sort&compress}

\newcommand{\sysn}{\left\{\begin{array}{rcl}}
\newcommand{\sysk}{\end{array}\right.}

\newtheorem{theorem}{Theorem}[section]
\newtheorem{lemma}[theorem]{Lemma}

\theoremstyle{example}

\newtheorem{proposition}[theorem]{Proposition}
\theoremstyle{definition}
\newtheorem{definition}[theorem]{Definition}
\newtheorem{remark}[theorem]{Remark}
\newtheorem{corollary}[theorem]{Corollary}

\journal{...}

\begin{document}

\title{Baire property of spaces of $[0,1]$-valued continuous functions}

\author{Alexander V. Osipov}

\address{Krasovskii Institute of Mathematics and Mechanics, \\ Ural Federal
 University, Ural State University of Economics, Yekaterinburg, Russia}

\ead{OAB@list.ru}

\author{Evgenii G. Pytkeev}

\address{Krasovskii Institute of Mathematics and Mechanics, Yekaterinburg, Russia}

\ead{pyt@imm.uran.ru}

\begin{abstract} A topological space $X$ is {\it Baire} if the intersection of any
sequence of open dense subsets of $X$ is dense in $X$. Let
$C_p(X,[0,1])$ denote the space of all continuous $[0,1]$-valued
functions on a Tychonoff space
 $X$ with the topology of pointwise convergence.

In this paper, we have obtained a characterization when the
function space $C_p(X,[0,1])$ is Baire for a Tychonoff space $X$
all separable closed subsets of which are $C^*$-embedded. In
particular, this characterization is true for normal spaces and,
hence, for metrizable spaces. Moreover, we obtained that the space
$C_p(X,[0,1])$ is Baire, if and only if, the space
$C_p(X,\mathbb{K})$ is Baire for a Peano continuum $\mathbb{K}$.

\end{abstract}


\maketitle 


\section{Introduction}

A topological space $X$ is {\it Baire} if the Baire Category
Theorem holds for $X$, i.e., the intersection of any sequence of
open dense subsets of $X$ is dense in $X$.

 A space is {\it meager} (or {\it of the first Baire category}) if it
can be written as a countable union of closed sets with empty
interior.  Clearly, if $X$ is a Baire space, then $X$ is not
meager. The reverse implication is in general not true. However,
it holds for every homogeneous space $X$ (see Theorem 2.3 in
\protect\cite{LM}).
 Being a Baire space is an important topological
property for a space and it is therefore natural to ask when
function spaces are Baire. The Baire property for continuous
mappings was first considered in \protect\cite{Vid}. Then a paper
\protect\cite{ZMc} appeared, where various aspects of this topic
were considered. In \protect\cite{ZMc}, necessary and, in some
cases, sufficient conditions on a space $X$ were obtained under
which the space $C_p(X)$ is Baire.

In general, it is not an easy task to characterize when a function
space has the Baire property. The problem for $C_p(X)$ was solved
independently by Pytkeev \cite{pyt1}, Tkachuk \cite{tk} and van
Douwen \protect\cite{vD}, and for $B_1(X)$ was solved by Osipov
\cite{Os}.



Since $C_p(X)$ is homeomorphic to $C_p(X,(0,1))$ and
$C_p(X,[0,1])$ contains the dense subspace $C_p(X,(0,1))$ then
$C_p(X,[0,1])$ is Baire for a Baire space $C_p(X)$. However, there
exists an example of a space $X$ such that $C_p(X,[0,1])$ is a
Baire space but $C_p(X)$ does not have the Baire property
(Ex.286,\cite{Tk}).

\medskip
In this paper, we have obtained a characterization when the
function space $C_p(X,[0,1])$ has the Baire property for a
Tychonoff space $X$ all separable closed subsets of which are
$C^*$-embedded. In particular, this characterization is true for
normal spaces and, hence, for metrizable spaces.

\section{Main definitions and notation}

Throughout this paper,  all spaces are assumed to be Tychonoff,
i.e., completely regular $T_1$-spaces.
 The set of positive integers is denoted by $\mathbb{N}$ and
$\omega=\mathbb{N}\cup \{0\}$. Let $\mathbb{R}$ be the real line,
we put $\mathbb{I}=[0,1]\subset \mathbb{R}$, and let $\mathbb{Q}$
be the rational numbers. Let $f:X\rightarrow \mathbb{R}$ be a
real-valued function, then $|| f ||= \sup \{|f(x)|: x\in X\}$,
$S(g,\epsilon)=\{f: \parallel g-f
\parallel<\epsilon\}$, $B(g,\epsilon)=\{f: \parallel g-f
\parallel\leq\epsilon\}$, where $g$ is a real valued function and
$\epsilon>0$. Let $V=\{f\in \mathbb{R}^X: f(x_i)\in V_i,
i=1,...,n\}$ where $x_i\in X$, $V_i\subseteq \mathbb{R}$ are
bounded intervals for $i=1,...,n$, then $supp V=\{x_1,...,x_n\}$ ,
$diam V=\max \{diam V_i : 1\leq i \leq n \}$.


Let $C_p(X)$ denote the space of all continuous real-valued
functions $C(X)$ on a space
 $X$ with the topology of pointwise convergence. Let
$C_p(X,[0,1])$ denote the subspace of $C_p(X)$ (of
$\mathbb{I}^X$), i.e., the set all continuous $[0,1]$-valued
functions on a Tychonoff space
 $X$ with the topology of pointwise convergence.

A mapping $f: X\rightarrow Y$ is said to be {\it irreducible} if
the only closed subset $Z$ of $X$, with $f(Z) = Y$ , is $Z = X$.

The set $A$ is {\it $C^*$-embedded} in $X$ if every bounded
continuous real-valued function on $A$ can be extended to a
bounded continuous real-valued function on $X$. Thus $X$ is normal
iff every closed subset is $C^*$-embedded.


Recall that two sets $A$ and $B$ are {\it functionally separated}
if there exists a continuous real-valued mapping $f$ on $X$ such
that $f[A]\subseteq \{0\}$ and $f[B]\subseteq\{1\}$. For the terms
and symbols that we do not define follow \cite{Eng}.

\section{Main results}

The following analogues of the Baire property are well known.


$\bullet$ A locally convex space is called a {\it Baire-like}
space if it is not the union of an increasing sequence of nowhere
dense, circled and convex sets.

$\bullet$ A locally convex space is called an {\it unordered
Baire-like} space if it is not the union of a sequence of nowhere
dense, circled and convex sets.

$\bullet$ A linear topological space is called {\it $W$-barreled}
if every closed, absorbing subset has non-empty interior
(\cite{Wil}, p.224).

\medskip

Clearly, a Baire locally convex space is a Baire-like space and an
unordered Baire-like space is Baire-like.

In \cite{pyt1}, Pytkeev proved that for a Tychonoff space $X$

$\bullet$ $C_p(X)$ is Baire if, and only if, $C_p(X)$ is
$W$-barreled (Th.5,\cite{pyt1}).

$\bullet$ $C_p(X)$ is barreled if, and only if, $C_p(X)$ is
Baire-like (Th.4,\cite{pyt1}).

There exists a Tychonoff space $X$ such that $C_p(X)$ is unordered
Baire-like but it is not Baire (Ex.1,\cite{pyt1}).

\medskip

In \protect\cite{ZMc}, it is proved that each homogeneous
non-meager space is Baire, hence, if $C_p(X)$ is a non-meager
space then it is Baire. Although the space $C_p(X, \mathbb{I})$ is
not homogeneous, this is also true for it. The following
proposition is some generalization the result that each
homogeneous non-meager space is Baire.

For an arbitrary topological space $X$, the autohomeomorphism
group $Aut(X)$ is a group consisting of all homeomorphism of $X$
onto itself. Denote by $G=Aut(X)$ and $Gx=\{g(x): g\in G\}$ for
some $x\in X$.

\begin{proposition}\label{pr1} Let $X$ be a non-meager
space and \, $\overline{Gx}=X$ for some $x\in X$. Then $X$ is
Baire.
\end{proposition}

\begin{proof} Note that a space $X$ is non-meager if and only if
there exists an open nonempty Baire subspace $\mathcal{U}$ of $X$.
Let $\mathcal{U}$ be a set of points of locally Baire ( A point
$x\in X$ is a point of locally Baire if there is a neighborhood
$V$ of $x$ such that $V$ is a Baire subspace of $X$).
$\mathcal{U}\neq \emptyset$ and $G\mathcal{U}=\mathcal{U}$. Then
$G(\bigcap \mathcal{U})\neq \emptyset$. Hence $Gx\subset
\mathcal{U}$ and $\overline{\mathcal{U}}=X$. It follows that $X$
is Baire.
\end{proof}

\begin{proposition}\label{pr2} Let $X=C_p(Y,\mathbb{I})$ , $f\equiv \frac{1}{2}$,
$G=Aut(X)$. Then $\overline{Gf}=X$.
\end{proposition}

\begin{proof} Let
$\widetilde{R}=\{-\infty\}\cup\mathbb{R}\cup\{+\infty\}\cong
\mathbb{I}$. For $\alpha\in \mathbb{R}$ we define
$f_{\alpha}:\widetilde{R}\rightarrow \widetilde{R}$ as follows

$$ f_{\alpha}(x):=\left\{
\begin{array}{lcl}
+\infty \, \, \, \, \, \, \, \, \, \, x=+\infty\\
x+\alpha \, \, \, \, \, \, x\in \mathbb{R}\\
-\infty \, \, \, \, \, \, \, \, \, \, x=-\infty.
\end{array}
\right.
$$

Note that $f_{\alpha}$ is a continuous function.

For each $g\in C_p(Y)$ we define $\varphi_g: X\rightarrow X$
($X=C_p(Y, \widetilde{R})$) as follows
$\varphi_g(h)(x)=f_{g(x)}(h(x))$. Let $G^*=\{\varphi_g: g\in
C_p(Y)\}\subset G$. Then $G^*f=C_p(Y)$ is dense in
$C_p(Y,\widetilde{R})$.

\end{proof}

\begin{corollary}  If $C_p(X,\mathbb{I})$ is a non-meager
space then it is Baire.
\end{corollary}

\begin{definition} We say that a space $X$ has the
{\it $\mathbb{B}$-property} if for every pairwise disjoint family
$\{\Delta_n: n\in \mathbb{N}\}$ of non-empty finite sets
$\Delta_n\subseteq X$, $\Delta_n=A_n\cup B_n$, $A_n\cap
B_n=\emptyset$, $n\in \mathbb{N}$, there exists a subsequence
$\{\Delta_{n_k}: k\in\mathbb{N}\}$ such that $\bigcup\{A_{n_k}:
k\in\mathbb{N}\}$ and $\bigcup\{B_{n_k}: k\in\mathbb{N}\}$ are
functionally separated.
\end{definition}

\begin{lemma}\label{lem1} Let $X$ be a space. The following
assertions are equivalent:

1. $X$ has the $\mathbb{B}$-property;

2. for every pairwise disjoint family $\{\Delta_n:
n\in\mathbb{N}\}$ of non-empty finite sets $\Delta_n\subseteq X$,
$n\in\mathbb{N}$ and $f: \bigcup\limits_{n\in \mathbb{N}}
\Delta_n\rightarrow \{0,1\}$ there are a subsequence
$\{\Delta_{n_k}: k\in\mathbb{N}\}$ and $\widetilde{f}\in
C(X,\mathbb{I})$ such that $\widetilde{f}\upharpoonright
\bigcup\{\Delta_{n_k}: k\in\mathbb{N}\}=f\upharpoonright
\bigcup\{\Delta_{n_k}: k\in\mathbb{N}\}$.

\end{lemma}

\begin{proof}
$(1)\Rightarrow(2)$. Let $\{\Delta_n: n\in \mathbb{N}\}$ be a
pairwise disjoint family of non-empty finite subsets of $X$ and
$f: \bigcup\limits_{n\in \mathbb{N}} \Delta_n\rightarrow \{0,1\}$.
Let $A_n=\Delta_n\cap f^{-1}(0)$, $B_n=\Delta_n\cap f^{-1}(1)$,
$n\in \mathbb{N}$. By (1), there is a sequence
$\{\Delta_{n_k}:k\in \mathbb{N}\}$ such that $\bigcup\{A_{n_k}:
k\in\mathbb{N}\}$ and $\bigcup\{B_{n_k}: k\in\mathbb{N}\}$  are
functionally separated, i.e., there is $\widetilde{f}\in
C(X,\mathbb{I})$ such that $\widetilde{f}(\bigcup\{A_{n_k}:
k\in\mathbb{N}\})=0$ and $\widetilde{f}(\bigcup\{B_{n_k}:
k\in\mathbb{N}\})=1$. Then, $\widetilde{f}$ is the required
function.

$(2)\Rightarrow(1)$. Let $\{\Delta_n: n\in\mathbb{N}\}$ be a
pairwise disjoint family of non-empty finite subsets of $X$ and
$\Delta_n=A_n\cup B_n$, $A_n\cap B_n=\emptyset$, $n\in\mathbb{N}$.
Define $f: \bigcup\limits_{n\in \mathbb{N}} \Delta_n\rightarrow
\{0,1\}$ such that $f(x)=0$ for $x\in \bigcup\{A_{n}:
n\in\mathbb{N}\}$ and $f(x)=1$ for $x\in \bigcup\{B_{n}:
n\in\mathbb{N}\}$. By (2), there are $\widetilde{f}\in
C(X,\mathbb{I})$ and $\{\Delta_{n_k}: k\in\mathbb{N}\}$ such that
$\widetilde{f}(\bigcup\{\Delta_{n_k}:
k\in\mathbb{N}\})=f(\bigcup\{\Delta_{n_k}: k\in\mathbb{N}\})$.
Then, $\{\Delta_{n_k}: k\in\mathbb{N}\}$ is the required sequence.

\end{proof}

\begin{definition} A space $X$ is called

$\bullet$ {\it separable $C^*$} if all separable closed subsets of
$X$ are $C^*$-embedded;

$\bullet$ {\it countable $C^*$} if all countable subsets of $X$
are $C^*$-embedded.
\end{definition}

Clear that a countable $C^*$ space is a separable $C^*$ space.
Note also that the class of separable $C^*$ spaces contains the
class of normal spaces.

\begin{lemma}\label{lem2} Let $X$ be a separable $C^*$ space. The following
assertions are equivalent:

1. $X$ has the $\mathbb{B}$-property;

2. for every pairwise disjoint family of non-empty finite sets
$\Delta_n\subseteq X$, $n\in\mathbb{N}$ and $f:
\bigcup\limits_{n\in \mathbb{N}} \Delta_n\rightarrow \mathbb{I}$
there are a subsequence $\{\Delta_{n_k}: k\in\mathbb{N}\}$ and
$\widetilde{f}\in C(X,\mathbb{I})$ such that
$\widetilde{f}\upharpoonright \bigcup\{\Delta_{n_k}:
k\in\mathbb{N}\}=f\upharpoonright \bigcup\{\Delta_{n_k}:
k\in\mathbb{N}\}$.

\end{lemma}

\begin{proof}

$(2)\Rightarrow(1)$. By Lemma \ref{lem1}, it is a trivial
implication.

$(1)\Rightarrow(2)$. Let $\gamma=\{\Delta_n: n\in\mathbb{N}\}$ be
a pairwise disjoint collection of non-empty finite subsets of $X$
and $f: \bigcup\limits_{n\in \mathbb{N}} \Delta_n\rightarrow
\mathbb{I}$. Let $\beta$ be a countable base in $\mathbb{I}$ that
is closed under finite unions and $\mathcal{U}=\{(U_n,V_n): n\in
\mathbb{N}\}$, $U_n,V_n\in \beta$ such that $\overline{U}_n\cap
\overline{V}_n=\emptyset$ for each $n\in \mathbb{N}$. Let
$A_n=\Delta_n\cap f^{-1}(\overline{U_1})$ and $B_n=\Delta_n\cap
f^{-1}(\overline{V_1})$ for $n\in \mathbb{N}$. By (1), there is
$N_1\subset \mathbb{N}$ such that $\bigcup\{A_{n}: n\in N_1\}$ and
$\bigcup\{B_{n}: n\in N_1\}$  are functionally separated. Let
$f_1=f\upharpoonright \bigcup_{n\in N_1} \Delta_n$. Then
$f_1^{-1}(\overline{U_1})$ and $f_1^{-1}(\overline{V_1})$ are
functionally separated. For each $i\in \mathbb{N}$ we choose
$N_i\subset \mathbb{N}$ such that $N_{i+1}\subset N_{i}$ and
$f_i=f\upharpoonright \bigcup_{n\in N_i} \Delta_n$ implies that
$f_i^{-1}(\overline{U_i})$ and $f_i^{-1}(\overline{V_i})$ are
functionally separated. Choose $n_i\in N_i$ such that
$n_i<n_{i+1}$ for each $i\in \mathbb{N}$. Let $g=f\upharpoonright
\bigcup_{i\in \mathbb{N}} \Delta_{n_i}$. Then
$g^{-1}(\overline{U_k})$ and $g^{-1}(\overline{V_k})$ are
functionally separated for each $k\in\mathbb{N}$.

Let $C,D\subset dom(g)$ such that $\overline{g(C)}\cap
\overline{g(D)}=\emptyset$. Then there is the pair $(U_n,V_n)\in
\mathcal{U}$ such that $\overline{g(C)}\subseteq U_n$,
$\overline{g(D)}\subseteq V_n$. Thus, $C$ and $D$ are functionally
separated. By Taimanov's Theorem (Th.1, \cite{Ta}), $g$ has a
continuous extension $\widetilde{g}$ on $\overline{\bigcup_{i\in
\mathbb{N}} \Delta_{n_i}}$. Since $X$ is a separable $C^*$-space
and $\overline{\bigcup_{i\in \mathbb{N}} \Delta_{n_i}}$ is a
closed set, there is a continuous extension $\widetilde{f}$ of
$\widetilde{g}$ on $X$.

\end{proof}

\begin{theorem}\label{th1} Let $X$ be a separable $C^*$ space. The following
assertions are equivalent:

1. $C_p(X,\mathbb{I})$ is Baire;

2. $C_p(X,\mathbb{I})$ is unordered Baire-like;

3. $C_p(X,\mathbb{I})$ is Baire-like;

4. $X$ has the $\mathbb{B}$-property.

\end{theorem}

\begin{proof} It is trivial that
$(1)\Rightarrow(2)\Rightarrow(3)$.

$(3)\Rightarrow(4)$. Let $\{\Delta_n: n\in \mathbb{N}\}$ be a
pairwise disjoint collection of non-empty finite subsets of $X$
such that  $\Delta_n=A_n\cup B_n$, $A_n\cap B_n=\emptyset$ for
each $n\in \mathbb{N}$. Without loss of generality we can assume
that $A_n\neq\emptyset$, $B_n\neq\emptyset$ for each $n\in
\mathbb{N}$. Consider $\mathcal{F}_m=\{f\in C(X,\mathbb{I}):
\sum\limits_{x\in A_n}f(x) -\sum\limits_{x\in B_n}f(x)\leq
|A_n|-\frac{1}{3}$ for each $n\geq m\}$ for each $m\in
\mathbb{N}$. It is easy to check that $\mathcal{F}_m$ is a closed
nowhere dense and convex set for each $m\in\mathbb{N}$. Since
$\mathcal{F}_m\subseteq \mathcal{F}_{m+1}$ ($m\in \mathbb{N}$) and
$C_p(X,\mathbb{I})$ is Baire-like then there is $g\in
C(X,\mathbb{I})\setminus \bigcup\limits_{m=1}^{\infty}
\mathcal{F}_m$. It follows that there is a sequence
$(n_k)_{k=1}^{\infty}$ such that $\sum\limits_{x\in A_{n_k}}g(x)
-\sum\limits_{x\in B_{n_k}}g(x)> |A_{n_k}|-\frac{1}{3}$ for each
$k\in \mathbb{N}$. Note that $g(x)\geq \frac{1}{2}$ for $x\in
A_{n_k}$, $k\in \mathbb{N}$ (otherwise $\sum\limits_{x\in
A_{n_k}}g(x) -\sum\limits_{x\in B_{n_k}}g(x)\leq \sum\limits_{x\in
A_{n_k}}g(x)<|A_{n_k}|-\frac{1}{2}$). Analogously, $g(x)\leq
\frac{1}{3}$ for $x\in B_{n_k}$, $k\in \mathbb{N}$. Hence,
$g(\bigcup\limits_{k=1}^{\infty} A_{n_k})\subseteq
[\frac{1}{2},1]$ and $g(\bigcup\limits_{k=1}^{\infty}
B_{n_k})\subseteq [0,\frac{1}{3}]$. Thus, the sets
$\bigcup\limits_{k=1}^{\infty} A_{n_k}$ and
$\bigcup\limits_{k=1}^{\infty} B_{n_k}$ are functionally
separated.

$(4)\Rightarrow(1)$.

{\it Claim} (I). For every set $\Delta=\{\Delta_n: n\in
\mathbb{N}\}$ of finite subsets of $X$ and $x\in X$ there is a
neighborhood $O(x)$ of $x$ such that $\Delta'=\{\Delta_n\in
\Delta: \Delta_n\cap O(x)=\emptyset\}$ is infinite.

Consider the sequence $\Delta'_n=A_n\cup B_n$ where
$A_n=\Delta_{2n}$, $B_n=\Delta_{2n-1}$ for $n\in\mathbb{N}$. By
(4), there is the sequence $(n_k)_{k=1}^{\infty}$ such that
$\bigcup\limits_{k=1}^{\infty} A_{n_k}$ and
$\bigcup\limits_{k=1}^{\infty} B_{n_k}$ are functionally separated
and, hence, $\overline{\bigcup\limits_{k=1}^{\infty} A_{n_k}}\cap
\overline{\bigcup\limits_{k=1}^{\infty} B_{n_k}}=\emptyset$. If
$x\not\in \overline{\bigcup\limits_{k=1}^{\infty} B_{n_k}}$ then
$O(x)=X\setminus \overline{\bigcup\limits_{k=1}^{\infty} B_{n_k}}$
as required. Analogously, if $x\not\in
\overline{\bigcup\limits_{k=1}^{\infty} A_{n_k}}$ then
$O(x)=X\setminus \overline{\bigcup\limits_{k=1}^{\infty}
A_{n_k}}$.

{\it Claim} (II). Let $\Delta=\{x_1,...,x_m\}\subseteq X$,
$\epsilon>0$, $F$ be a closed nowhere dense subset of
$C_p(X,\mathbb{I})$. Then there is an open basis set $V$ such that

a). $supp V\cap \Delta=\emptyset$;

b). for every $f\in V$ there is $\varphi_1: supp V\cup
\Delta\rightarrow \mathbb{I}$ such that
$||\varphi_1-f\upharpoonright supp V\cup \Delta||<\epsilon$ and
$\varphi_1\not\in \overline{\pi_{supp V\cup \Delta}F}$.

Let $O$ be an open set in $\mathbb{I}^{\Delta}$. Then
$W(\Delta,O):=\{f\in C(X,\mathbb{I}): (f(x_i))_{i=1}^m\in O\}$ is
an open set in $C_p(X,\mathbb{I})$. Choose a cover $\{O_i:
i=1,...,n\}$ of $I^\Delta$ such that $diam O_i<\epsilon$ for each
$i=1,...,n$. By induction, we construct open basis sets $W_i$,
$V_i$, $1\leq i \leq n$ such that $supp W_i=\Delta$, $i=1,...,n$,
$V_{i+1}\subseteq V_i$, $supp V_{i+1}\supseteq supp V_i$,
$i=1,...,n-1$, $supp V_i\cap \Delta=\emptyset$, $W_{i+1}\cap
V_{i+1}\subseteq (W(\Delta, O_{i+1})\cap V_i)\setminus F$.

Since $W(\Delta, O_1)$ is a non-empty open set, there are open
basis sets $W_1$ and $V_1$ such that $W_1\cap V_1\subseteq
W(\Delta, O_1)\setminus F$, $supp W_1=\Delta$ and $supp V_1\cap
\Delta=\emptyset$.

Suppose that the sets $W_i$, $V_i$ ($1\leq i\leq k<n$) are
constructed. Then $W(\Delta, O_{k+1})\cap V_k$ is a non-empty open
set. Hence, there are open basis sets $W_{k+1}$, $V_{k+1}$ such
that $W_{k+1}\cap V_{k+1}\subseteq (W(\Delta, O_{k+1})\cap
V_k)\setminus F$, $supp W_{k+1}=\Delta$, $supp V_{k+1}\supseteq
supp V_k$, $V_{k+1}\subseteq V_k$. The sets $W_i$, $V_i$ ($1\leq
i\leq n$) are constructed.

Let $V:=V_n$. Then $supp V\cap \Delta=\emptyset$. If $f\in V$ then
there is $O_{i_0}$ such that $(f(x_i))_{i=1}^m\in O_{i_0}$. Let
$W_{i_0}=\{\varphi\in C(X,\mathbb{I}): \varphi(x_i)\in U_i,
i=1,...,m\}$. Define $\varphi_1: suppV\cup \Delta\rightarrow I$ as
follows $\varphi_1(x_i)\in U_i$, $i=1,...,m$, $\varphi_1(x)=f(x)$
for $x\in suppV$. Since $W_{i_0}\cap V_{i_0}\subseteq W(\Delta,
O_{i_0})$ and $V_n\subseteq V_{i_0}$, then $W_{i_0}\cap
V_n\subseteq W(\Delta, O_{i_0})$. Then,
$||\varphi_1-f\upharpoonright \Delta\cup supp
V||=||\varphi_1\upharpoonright  \Delta -f\upharpoonright
\Delta||\leq diam O_{i_0}<\epsilon$. Since $(W_{i_0}\cap V_n)\cap
F=\emptyset$, then $\pi_{\Delta\cup suppV} F\not\ni \varphi_1$.
Thus, the proposition (II) is proved.

Assume that $C_p(X,\mathbb{I})$ is meager, i.e.,
$C_p(X,\mathbb{I})=\bigcup\limits_{n=1}^{\infty} F_n$,
$F_n\subseteq F_{n+1}$ where $F_n$ is a closed nowhere dense
subset of $C_p(X,\mathbb{I})$ for each $n\in \mathbb{N}$.

By (II), we construct open basis sets $V_n$ ($n\in \mathbb{N}$).
Let $V_1$ be arbitrary open basis set such that $V_1\cap
F_1=\emptyset$ and $diam V_1<1$. If $V_1$,...,$V_n$ are
constructed then, by (II), for $F=F_{n+1}$,
$\Delta=\bigcup\limits_{i=1}^{n} suppV_i$,
$\epsilon=\frac{1}{2^{n-1}}$ we construct the set $V_{n+1}$ such
that $diam V_{n+1}<\frac{1}{2^n}$. Let $\Delta_n=supp V_n$, $n\in
\mathbb{N}$. Note that $\Delta_n\cap \Delta_{n'}=\emptyset$ for
$n\neq n'$. By Lemma \ref{lem2}, there is an infinite subset
$N_0\subseteq \mathbb{N}$ such that $\bigcap\{V_n: n\in N_0\}\neq
\emptyset$. Let $f_0\in \bigcap\{V_n: n\in N_0\}$. By induction,
we construct sequences $f_n\in C(X,\mathbb{I})$, $\epsilon_n$,
$0<\epsilon_n<\frac{1}{2^n}$, $\epsilon_{n+1}<\epsilon_{n}$,
$N_n\subset \mathbb{N}$, $N_{n+1}\supseteq N_n$ such that

$\alpha$). $B(f_n, \epsilon_{n})\cap F_n=\emptyset$, $n\in
\mathbb{N}$,

$\beta$). $||f_n-f_k||<\epsilon_n$, $k>n$, $k\in\mathbb{N}$.

Since $f_0\in V_{n_1}$, $n_1\in N_0$, $n_1>1$ then, by (II), there
is $\varphi_1: \bigcup\limits_{i\leq n_1} \Delta_i\rightarrow
\mathbb{I}$ such that $||f_0\upharpoonright \bigcup\limits_{i\leq
n_1} \Delta_i -\varphi_1||\leq \frac{1}{2^{n_1-1}}\leq
\frac{1}{2}$ and $\overline{\pi_{\bigcup\limits_{i\leq n_1}
\Delta_i} F_{n_1}}\not\ni \varphi_1$.

By (I), there is an open set $O(\bigcup\limits_{i\leq n_1}
\Delta_i)$ such that $N_1=\{n\in N_0: \Delta_n\cap
O(\bigcup\limits_{i\leq n_1} \Delta_i)=\emptyset\}$ is infinite.
Choose $\epsilon_1<\frac{1}{2}$ and $f_1\in C(X,\mathbb{I})$ such
that $B(\varphi_1,\epsilon_1)\cap F_{n_1}=\emptyset$  and
$f_1\upharpoonright (X\setminus O(\bigcup\limits_{i\leq n_1}
\Delta_i))=f_0\upharpoonright (X\setminus O(\bigcup\limits_{i\leq
n_1} \Delta_i))$, $f_1\upharpoonright \bigcup\limits_{i\leq n_1}
\Delta_i=\varphi_1$, $||f_0-f_1||\leq \frac{1}{2^{n_1-1}}$.

Let $f_1$,..., $f_k$ are constructed. We have $N_0\supseteq
N_1\supseteq...\supseteq N_k$, $\epsilon_n$ for $n\leq k$ and
$n_{i+1}\in N_i$, $n_i<n_{i+1}$ for each $i\leq k-1$.

Choose $n_{k+1}\in N_k$, $n_{k+1}>n_k$. Then $f_k\in V_{n_{k+1}}$,
by (II), there is $\varphi_{k+1}: \bigcup\limits_{i\leq n_{k+1}}
\Delta_i\rightarrow \mathbb{I}$ such that $||f_k\upharpoonright
\bigcup\limits_{i\leq n_{k+1}} \Delta_i-\varphi_{k+1}||\leq
\frac{1}{2^{n_{k+1}-1}}\leq\frac{1}{2^{k+1}}$ and
$\overline{\pi_{\bigcup\limits_{i\leq n_{k+1}}
\Delta_i}F_{n_{k+1}}}\not\ni\varphi_{k+1}$. By (I), there is an
open set $O(\bigcup\limits_{i\leq n_{k+1}} \Delta_i)$ such that
the set $N_{k+1}=\{n\in N_k: \Delta_n\cap O(\bigcup\limits_{i\leq
n_{k+1}} \Delta_i)=\emptyset\}$ is infinite. Choose
$\epsilon_{k+1}<\frac{1}{2^{k+1}}$, $\epsilon_{k+1}<\epsilon_k$
and $f_{k+1}\in C(X,\mathbb{I})$ such that
$B(\varphi_{k+1},\epsilon_1)\cap F_{n_{k+1}}=\emptyset$ and
$f_{k+1}\upharpoonright X\setminus O(\bigcup\limits_{i\leq
n_{k+1}} \Delta_i)=f_k\upharpoonright X\setminus
O(\bigcup\limits_{i\leq n_{k+1}} \Delta_i)$,
$f_{k+1}\upharpoonright \bigcup\limits_{i\leq n_{k+1}}
\Delta_i=\varphi_{k+1}$, $||f_k-f_{k+1}||<\frac{1}{2^{k+1}}$.
Thus, functions $f_k$, $k\in\mathbb{N}$, satisfying conditions
($\alpha$) and ($\beta$) are constructed.

By ($\beta$), the sequence $(f_k)$ is fundamental in the metric
space $(C(X,\mathbb{I}), d)$ where $d$ is a uniform metric, i.e.,
$d(f,g)=\sup\{|f(x)-g(x)|: x\in X\}$. Then there is $g\in
C(X,\mathbb{I})$ such that $g$ is a uniform limit of $(f_k)$. But,
by conditions ($\alpha$) and ($\beta$), $g\not\in
C(X,\mathbb{I})$. Contradiction.
\end{proof}

\begin{remark}\label{rem1} If $C_p(X,\mathbb{I})$ is Baire then $X$ has
the $\mathbb{B}$-property for any Tychonoff space $X$.
\end{remark}

\begin{corollary} Let $X$ be a normal space. The following
assertions are equivalent:

1. $C_p(X,\mathbb{I})$ is Baire;

2. $C_p(X,\mathbb{I})$ is unordered Baire-like;

3. $C_p(X,\mathbb{I})$ is Baire-like;

4. $X$ has the $\mathbb{B}$-property.

\end{corollary}

Recall that a topological space $X$ is {\it scattered} if every
nonempty subset $A$ of $X$ contains a point isolated in $A$.

By Remark \ref{rem1}, we have the following result.

\begin{corollary} Let $X$ be a space and $B$ be an infinite
compact subset of $X$. Then $C_p(X,\mathbb{I})$ is meager.
\end{corollary}

\begin{proof}
a). Suppose that $B$ is not scattered. Then there are a compact
space $K$, $K\subseteq B$, and a irreducible continuous mapping
$\varphi: K\rightarrow \mathbb{I}$, $\varphi(K)=\mathbb{I}$.
Choose a disjoint family $\{\Delta_n: n\in \mathbb{N}\}$ of finite
subsets of $X$ such that

1. $\Delta_n=A_n\cup B_n$ for $n\in \mathbb{N}$;

2. $\{\varphi(\Delta_n): n\in\mathbb{N}\}$ is a disjoint family;

3. $\varphi(A_n)\cap \varphi(B_n)=\emptyset$ for $n\in
\mathbb{N}$;

3. $\varphi(A_n)$, $\varphi(B_n)$ are $\frac{1}{n}$-nets in
$\mathbb{I}$, i.e., for each $x\in \mathbb{I}$ there exists $a\in
\varphi(A_n)$ ($a\in \varphi(B_n)$) such that $|x-a|<\frac{1}{n}$
for each $n\in \mathbb{N}$.

Assume that $\{\Delta_{n_k}: k\in\mathbb{N}\}$ is a subsequence of
$\{\Delta_{n}: n\in\mathbb{N}\}$ such that
$\bigcup\limits_{k=1}^{\infty} A_{n_k}$ and
$\bigcup\limits_{k=1}^{\infty} B_{n_k}$ are functionally
separated. Then $\bigcup\limits_{k=1}^{\infty} A_{n_k}$ and
$\bigcup\limits_{k=1}^{\infty} B_{n_k}$ are separated, i.e., there
exists disjoint open sets in $X$ containing these sets. But,
$\varphi(\bigcup\limits_{k=1}^{\infty} A_{n_k})$ and
$\varphi(\bigcup\limits_{k=1}^{\infty} B_{n_k})$ are dense subsets
of $\mathbb{I}$. Since $\varphi$ is irreducible then
$\bigcup\limits_{k=1}^{\infty} A_{n_k}$ and
$\bigcup\limits_{k=1}^{\infty} B_{n_k}$ are dense subsets of $K$.
Thus, $\bigcup\limits_{k=1}^{\infty} A_{n_k}$ and
$\bigcup\limits_{k=1}^{\infty} B_{n_k}$ are not separated.
Contradiction.

b). Let $B$ be a scattered compact space. There exists a
convergent sequence $(x_n)\rightarrow x_0$, $x_n\in B$, $n\in
\omega$. Let $A_n=\{x_{2n}\}$, $B_n=\{x_{2n-1}\}$,
$\Delta_n=A_n\cup B_n$, $n\in \mathbb{N}$. It remains to note that
$x_0\in \overline{\bigcup\limits_{k=1}^{\infty}
A_{n_k}}\cap\overline{\bigcup\limits_{k=1}^{\infty} B_{n_k}}$
whenever a subsequence $(n_k)_{k=1}^{\infty}$ of $\mathbb{N}$.

\end{proof}

\begin{theorem}\label{th7} Let $X$ be a space all countable subsets of which are
scattered and $C^*$-embedded. Then $C_p(X,\mathbb{I})$ is Baire.
\end{theorem}

\begin{proof}

{\it (I)}. For every disjoint sequence $(\Delta_n)_{n=1}^{\infty}$
of non-empty finite sets of $X$ there exists an infinite
subsequence $(\Delta_{n_k})_{k=1}^{\infty}$ such that the set
$\bigcup\limits_{k=1}^{\infty}\Delta_{n_k}$ is discrete.

Let $S=\bigcup\limits_{n=1}^{\infty}\Delta_{n}$. Then $S$ is
scattered. Let $\alpha=in(S)$ where $in(S)$ is an index of
scattered of $S$.

If $\alpha=0$ then $S$ is discrete.

Let (I) be a true for $S$ where $in(S)<\beta$ and $\beta$ is a
limit cardinal. Then $S=\bigcup\limits_{i=1}^{\infty} U_i$ where
$U_i\subseteq U_{i+1}$,  $U_i$ is an infinite open set in $S$ and
$in(U_i)<\beta$ for each $i\in \mathbb{N}$. By induction, there is
a sequence $N_1\supseteq N_2\supseteq ...$ of infinite subsets of
$\mathbb{N}$ such that $\bigcup \{U_i\cap \Delta_n: n\in N_i\}$ is
discrete for each $i\in \mathbb{N}$. Choose $n_k\in N_k$,
$n_{k+1}>n_k$, $k\in \mathbb{N}$. Then
$\bigcup\limits_{k=1}^{\infty} \Delta_{n_k}$ is discrete.

Let $\beta=\beta_0+1$. Then $S=U\cup D$ where $U\subseteq S$ is
open, $in(U)=\beta_0$, $D=S\setminus U$ is discrete in $S$. By
induction, there is an infinite subset $N_1$ of $\mathbb{N}$ such
that $T=\bigcup \{U\cap \Delta_n: n\in N_1\}$ is discrete. We can
assume that $T$ is infinite. Since $T$ is countable, discrete and
$C^*$-embedded then $\overline{A}\cap\overline{B}=\emptyset$
whenever $A,B\subset T$, $A\cap B=\emptyset$. It follows that for
each $x\in D$ there is a neighborhood $O(x)$ of $x$ such that the
set $N'=\{n\in N_1: O(x)\cap \Delta_n=\emptyset\}$ is infinite.
Since $D$ is discrete then we can construct (by induction) a
sequence of subsets $N_1\supseteq N_2\supseteq...$,
$\Delta_{n_k}$, $k\in \mathbb{N}$, $n_k\in N_k$ and open sets
$O(\Delta_{n_k})$ such that $\Delta_{n_k}\subseteq
O(\Delta_{n_k})$, $O(\Delta_{n_k})\cap \Delta_{n_j}=\emptyset$,
$j>k$, $k\in \mathbb{N}$. Then the sequence $\{\Delta_{n_k}: k\in
\mathbb{N}\}$ required. The proposition (I) is proved.

Let $\{\Delta_{n}: n\in \mathbb{N}\}$ be an arbitrary disjoint
sequence of finite subsets of $X$ and
$f:\bigcup\limits_{n=1}^{\infty}\rightarrow \mathbb{I}$ be an
arbitrary function. By (I), there is a sequence $\{\Delta_{n_k}:
k\in \mathbb{N}\}$ such that $\bigcup\limits_{k=1}^{\infty}
\Delta_{n_k}$ is discrete, hence, $f\upharpoonright
\bigcup\limits_{k=1}^{\infty} \Delta_{n_k}$ is continuous. Since
any countable subset of $X$ is $C^*$-embedded, there is
$\widetilde{f}\in C(X,\mathbb{I})$ such that
$\widetilde{f}\upharpoonright \bigcup\limits_{k=1}^{\infty}
\Delta_{n_k}=f\upharpoonright \bigcup\limits_{k=1}^{\infty}
\Delta_{n_k}$. By Lemma \ref{lem1} and Theorem \ref{th1},
$C_p(X,\mathbb{I})$ is Baire.
\end{proof}

Recall that a subset $A\subseteq X$ {\it bounded} \, if, for any
$f\in C(X)$, the set $f(A)$ is bounded in $\mathbb{R}$. Note that
if $C_p(X)$ is a Baire space then every bounded subset of $X$ is
finite (\cite{pyt1} and Ex.284 in \cite{Tk}). In \cite{Sh},
D.Shakhmatov constructed a pseudocompact Tychonoff space $X$ all
countable subsets of which are closed and $C^*$-embedded. Thus, by
Theorem \ref{th7}, the Shakhmatov's example is an example of an
infinite pseudocompact (bounded in itself) Tychonoff space $X$
such that $C_p(X,\mathbb{I})$ is Baire and $C_p(X)$ is meager.

\begin{proposition} There exists an infinite countably compact Tychonoff space $X$ such
that $C_p(X,\mathbb{I})$ is Baire.
\end{proposition}

In \cite{JuWe}, it is constructed a countably compact space
$X\subseteq \beta \mathbb{N}$ such that all countable subsets of
$X$ are scattered. It is well known that all countable subsets of
$\beta \mathbb{N}$ are $C^*$-embedded \cite{GJ}. Hence, by Theorem
\ref{th7}, $C_p(X,\mathbb{I})$ is Baire.

\medskip

Recall that a mapping $\varphi: K\rightarrow M$ is called {\it
almost open} if $Int \varphi(V)\neq \emptyset$ whenever non-empty
open subset $V$ of $K$. Note that irreducible mappings are almost
open mappings that are defined on compact spaces. Also note that
if $\varphi_{\alpha}:K_{\alpha}\rightarrow M_{\alpha}$ ($\alpha\in
A$) are surjective almost open mappings then the product mapping
$\prod\limits_{\alpha\in A}\varphi_{\alpha}:
\prod\limits_{\alpha\in A} K_{\alpha}\rightarrow
\prod\limits_{\alpha\in A} M_{\alpha}$ is also almost open.

\medskip

\begin{lemma}\label{lem5} Let $\psi:P\rightarrow L$ be a surjective continuous almost open
mapping and $E\subseteq P$ be a dense non-meager (Baire) subspace
in $P$. Then $\psi(E)$ is non-meager (Baire).
\end{lemma}

\begin{proof}
Claim that the preimage $\psi^{-1}(Z)$ is dense in $P$ for every
dense set $Z$ in $L$.

Suppose, contrary to our claim, that $V=P\setminus
\overline{\psi^{-1}(Z)}$ is a non-empty open set in $P$. Since
$\psi$ is almost open then $Int$ $\psi(V)\neq \emptyset$. But
$Int$ $\psi(V)\cap Z=\emptyset$, contradiction.

Let $G_i$, $i\in\mathbb{N}$ be open dense subset of $L$. Then
$\psi^{-1}(G_i)$, $i\in \mathbb{N}$ is open dense subsets of $P$.
Since $E$ is a dense non-meager (Baire) subspace in $P$ then
$\bigcap\limits_{i=1}^{\infty}\psi^{-1}(G_i)\cap E\neq\emptyset$
(it is dense in $E$). Hence
$\psi(\bigcap\limits_{i=1}^{\infty}\psi^{-1}(G_i)\cap
E)=\bigcap\limits_{i=1}^{\infty}G_i\cap \psi(E)\neq\emptyset$ (it
is dense in $\psi(E)$).
\end{proof}

\begin{lemma}\label{lem6} Let $\psi: K\rightarrow M$ be a surjective continuous almost open
mapping, and $C_p(X,K)$ is a non-meager (Baire) dense subspace of
$K^X$. Then $C_p(X,M)$ is non-meager (Baire).
\end{lemma}

\begin{proof} The mapping $\psi^X: K^X\rightarrow M^X$ is a continuous almost open
mapping. Since $C_p(X,K)$ is dense in $K^X$ and is non-meager
(Baire) then, by Lemma \ref{lem5}, $\psi^X(C_p(X,K))$ is
non-meager (Baire). Note that $\psi^X(C_p(X,K))=\{\psi\circ f:
f\in C_p(X,K)\}\subseteq C_p(X,M)$. Thus, $C_p(X,M)$ contains a
dense non-meager (Baire) subspace, hence, it also is non-meager
(Baire).

\end{proof}

\begin{lemma}\label{lem7} (\cite{arhPon}) Let $L$ be a compact space without
isolated points, $T$ be the Cantor set. Then there exists a
surjective continuous irreducible mapping $\psi:T\rightarrow L$.
\end{lemma}

Recall that a {\it Peano continuum}  is called a connected and
locally connected metric compact space. Denote by $\mathbb{K}$ a
Peano continuum $X$ such that $|X|>1$.

\begin{lemma}\label{lem8} There exists a surjective continuous almost open mapping $\psi:
\mathbb{I}\rightarrow \mathbb{K}$.
\end{lemma}

\begin{proof} Let $\rho$ be a metric in $\mathbb{K}$.

Note that true the following results \cite{Kr}:

1). for every $\epsilon>0$ there is $\delta>0$ such that if
$x,y\in \mathbb{K}$ and $\rho(x,y)<\delta$ there is  a path
$L(x,y)$ such that $diam L(x,y)<\epsilon$ ($L(x,y):=f(\mathbb{I})$
where $f:\mathbb{I}\rightarrow X$ is a continuous mapping such
that $f(0)=x$ and $f(1)=y$).

2). for every $x\in \mathbb{K}$ there is a base $\{O(x)\}$ of
neighborhoods of $x$ such that $\overline{O(x)}$ is a Peano
continuum for every $O(x)\in \{O(x)\}$.

Renumber the rational numbers $\mathbb{Q}\cap \mathbb{I}$ as
$\{r_n: n\in \mathbb{N}\}$ in the segment $\mathbb{I}$. By
induction we construct a disjoint sequence $\{T_i: i\in
\mathbb{N}\}$ of copies of the Cantor set and a sequence $\{f_i:
i\in \mathbb{N}\}$ of surjective continuous mappings
$f_i:\mathbb{I}\rightarrow \mathbb{K}$ such that

3). $diam T_i\rightarrow 0$,

4). $conv T_i\cap T_j=\emptyset$ for $i>j$,

5). $\{r_i: i=1,...,n\}\subseteq \bigcup\limits_{i=1}^n T_i$ for
$n\in \mathbb{N}$,

6). $diam f_n(conv T_n)<\frac{1}{2^{n+1}}$ for $n\in \mathbb{N}$,

7). $f_1(T_1)=\mathbb{K}$, $f_i\upharpoonright T_i$ is irreducible
and $f_i(T_i)=\overline{Int f_i(T_i)}$ for $i\in \mathbb{N}$,

8). $f_{n+1}\upharpoonright \bigcup\limits_{i=1}^n
T_i=f_n\upharpoonright \bigcup\limits_{i=1}^n T_i$ for $n\in
\mathbb{N}$,

9). $||f_{n+1}-f_n||<\frac{1}{2^n}$ for $n\in \mathbb{N}$.

Construct $f_1$ and $T_1$. Let $T_1$ be an arbitrary Cantor set in
the segment $\mathbb{I}$ and $r_1\in T_1$. By Lemma \ref{lem7},
there is a surjective continuous irreducible mapping
$\varphi_0:T_1\rightarrow \mathbb{K}$. Let $(a_i,b_i)$, $i\in
\mathbb{N}$ be adjacent intervals of $T_1$. Then
$\lim\limits_{i\rightarrow \infty} |a_i-b_i|=0$. Since $\varphi_0$
is uniform continuous, by the condition 1), there are paths
$L(\varphi_0(a_i),\varphi_0(b_i))$ such that $diam
L(\varphi_0(a_i),\varphi_0(b_i))\rightarrow 0$. Let
$\varphi_i:[a_i,b_i]\rightarrow L(\varphi_0(a_i),\varphi_0(b_i))$
be homeomorphisms such that $\varphi_i(a_i)=\varphi_0(a_i)$,
$\varphi_i(b_i)=\varphi_0(b_i)$, $i\in \mathbb{N}$. Then
$f_1=\bigcup\limits_{i=0}^{\infty} \varphi_i$ is a required
mapping.

Assume that $f_1$, ..., $f_n$ and $T_1$,..., $T_n$ satisfying
conditions 3).-9). are constructed. Let $k=\min\{m: r_m\not\in
\bigcup\limits_{i=1}^n T_i\}$. Consider a neighborhood
$O(f_n(r_k))$ of $f_n(r_k)$ such that
$\overline{O(f_n(r_k))}<\frac{1}{2^{n+1}}$ and
$\overline{O(f_n(r_k))}$ is a Peano continuum. Since $f_n$ is
continuous there is $\epsilon>0$ such that
$I_{n+1}=[r_k-\epsilon,r_k+\epsilon]\cap (\bigcup\limits_{i=1}^n
T_i)=\emptyset$ and $f_n(I_{n+1})\subseteq O(f_n(r_k))$. Let
$T_{n+1}$ be a copy of the Cantor set such that $T_{n+1}\subseteq
I_{n+1}$, $r_k\in T_{n+1}$. Construct a surjective continuous
mapping $\psi: I_{n+1}\rightarrow \overline{O(f_n(r_k))}$ such
that $\psi\upharpoonright T_{n+1}: T_{n+1}\rightarrow
\overline{O(f_n(r_k))}$ is surjective irreducible and
$\psi(r_k-\epsilon)=f_n(r_k-\epsilon)$,
$\psi(r_k+\epsilon)=f_n(r_k+\epsilon)$. Let $f_{n+1}(x)=f_n(x)$
for $x\in [0,r_k-\epsilon]\cup [r_k+\epsilon,1]$ and
$f_{n+1}(x)=\psi(x)$ for $x\in [r_k-\epsilon,r_k+\epsilon]$. Clear
that $f_{n+1}$ is a continuous mapping and the conditions 3).-9).
holds.By the condition 9), there is the mapping
$f=\lim\limits_{n\rightarrow \infty} f_n$ and $f$ is continuous.
By the conditions 7) and 8), $f(\mathbb{I})=\mathbb{K}$. By the
condition 5), $\bigcup\limits_{i=1}^{\infty} T_i$ is dense in
$\mathbb{I}$. It follows that ( and by the condition 7)) $f$ is
almost open.
\end{proof}

\begin{theorem} Let $X$ be a separable $C^*$-space. Then the following assertions are equivalent:

1. $C_p(X,\mathbb{I})$ is Baire;

2. $C_p(X,\mathbb{K})$ is Baire.

\end{theorem}

\begin{proof}
$(1)\Rightarrow(2)$. By Lemma \ref{lem8}, there is a surjective
continuous almost open mapping $\psi: \mathbb{I}\rightarrow
\mathbb{K}$. Then, by Lemma \ref{lem6}, $C_p(X,\mathbb{K})$ is
Baire.

$(2)\Rightarrow(1)$. Since $|\mathbb{K}|>1$ there exists a
surjective continuous mapping $\psi: \mathbb{K}\rightarrow
\mathbb{I}$. Let $\{\Delta_n: n\in \mathbb{N}\}$ be a disjoint
sequence of non-empty finite sets of $X$ and $\Delta_n=A_n\cup
B_n$ where $A_n\cap B_n=\emptyset$ for each $n\in \mathbb{N}$. Let
$M_n=\{f\in C(X,\mathbb{K}): \psi f(A_n)\subseteq
[0,\frac{1}{3}]$, $\psi f(B_n)\subseteq [\frac{2}{3},1]\}$ for
$n\in \mathbb{N}$, $F_m=\bigcap\limits_{n\geq
m}\{C(X,\mathbb{K})\setminus M_n\}$, $m\in \mathbb{N}$. Since
$\mathbb{K}$ is arcwise connected then $F_m$ is nowhere dense for
each $m\in \mathbb{N}$. Let $f_0\in C(X,\mathbb{K})\setminus
\bigcup\limits_{m=1}^{\infty} F_m$. Then there is a subsequence
$\{M_{n_k}: k\in \mathbb{N}\}$ such that $f_0\in
\bigcap\limits_{k=1}^{\infty} M_{n_k}$. Hence, $\psi
f_0(\bigcup\limits_{k=1}^{\infty} A_{n_k})\subseteq
[0,\frac{1}{3}]$, $\psi f_0(\bigcup\limits_{k=1}^{\infty}
B_{n_k})\subseteq (\frac{2}{3},1]$. By Theorem \ref{th1},
$C(X,\mathbb{I})$ is Baire.
\end{proof}

\begin{corollary} Let $X$ be a normal space and $\mathbb{K}$ be a Peano continuum. Then $C_p(X,\mathbb{I})$ is Baire if, and only if,
$C_p(X,\mathbb{K})$ is Baire.
\end{corollary}

\medskip

{\bf Acknowledgements.} We would like to thank Evgenii Reznichenko
for several valuable comments on Propositions \ref{pr1} and
\ref{pr2}.

\bibliographystyle{model1a-num-names}
\bibliography{<your-bib-database>}

\end{document}